\documentclass[12pt,reqno]{amsart}
 \usepackage[foot]{amsaddr}
 \usepackage{orcidlink}
 \usepackage[margin=3cm]{geometry}
  \usepackage[all,2cell,arrow]{xy}
  \usepackage{amsfonts}
  \usepackage{amsthm}
  \usepackage{amsmath}
  \usepackage{quiver}
  \usepackage{amssymb}
  \usepackage{tikz-cd}
\usepackage{mathtools}
\usepackage{tensor}
 \numberwithin{equation}{section}
 \usepackage{color}
\usepackage{graphicx}
  \usepackage{hyperref}
\usepackage{enumerate,xspace}

\UseAllTwocells
\CompileMatrices

\renewcommand{\epsilon}{\varepsilon}
\renewcommand{\phi}{\varphi}

\newcommand{\ca}{\ensuremath{\mathcal A}\xspace}
\newcommand{\cb}{\ensuremath{\mathcal B}\xspace}
\newcommand{\cc}{\ensuremath{\mathcal C}\xspace}

\newcommand{\Cat}{\ensuremath{\mathbf{Cat}}\xspace}
\newcommand{\twocat}{\ensuremath{\mathbf{2}\textnormal{-}\mathbf{Cat}}\xspace}

\DeclareMathOperator{\Lex}{\mathbf{Lex}}
\DeclareMathOperator{\Reg}{\mathbf{Reg}}
\DeclareMathOperator{\BCop}{\mathbf{BCop}}
\DeclareMathOperator{\PhiCol}{\mathbf{\Phi\textnormal{-}Col}}
\DeclareMathOperator{\PhiLim}{\mathbf{\Phi\textnormal{-}Lim}}

\DeclareMathOperator{\TAlg}{\textnormal{T-Alg}}
\DeclareMathOperator{\PTAlg}{\textnormal{Ps-T-Alg}}
\DeclareMathOperator{\Kan}{\textnormal{LInj}}

\def\1c#1{\stackrel{#1}{\to}}

  \newtheorem{proposition}{Proposition}[section]
  \newtheorem{lemma}[proposition]{Lemma}
  
  \newtheorem{corollary}[proposition]{Corollary}
  \newtheorem{theorem}[proposition]{Theorem}

  \theoremstyle{definition}
  
  \newtheorem{notation}[proposition]{Notation}
  \newtheorem{example}[proposition]{Example}
  \newtheorem{examples}[proposition]{Examples}

  \theoremstyle{remark}
  \newtheorem{remark}[proposition]{Remark}

  \newcounter{c}
  \renewcommand{\[}{\setcounter{c}{1}$$}
  \newcommand{\etyk}[1]{\vspace{-7.4mm}$$\begin{equation}\Label{#1}
  \addtocounter{c}{1}}
  \renewcommand{\]}{\ifnum \value{c}=1 $$\else \end{equation}\fi}
\begin{document}
\leftmargini=2em \title{A recognition criterion for lax-idempotent pseudomonads}
\author{John Bourke}
\orcidlink{0000-0002-4617-7798}
\address{Department of Mathematics and Statistics, Masaryk University, Kotl\'a\v rsk\'a 2, Brno 61137, Czech Republic}
\email{bourkej@math.muni.cz}
\subjclass[2020]{Primary: 18N10, 18N15, 18C15}
\date{\today}
\maketitle

\maketitle
\begin{abstract}
We describe a simple criterion which makes it easy to recognise when a pseudomonad is lax-idempotent.  The criterion concerns the behaviour of colax bilimits of arrows --- certain comma objects --- and is easy to verify in examples.  Building on this, we obtain a new characterisation of lax-idempotent pseudomonads on $2$-categories with colax bilimits of arrows.  
\end{abstract}

\section{Introduction}
Lax-idempotent pseudomonads are lovely objects in $2$-category theory which can be described in several ways.  Firstly, they are pseudomonads $T$ for which structure is adjoint to unit, meaning that the multiplication $\mu \colon T^2 \to T$ is left adjoint to the unit $\eta_T \colon T \to T^2$.  In fact, a whole string of adjunctions $T\eta \dashv \mu \dashv \eta_T$ then  follows.  The concept is due to Kock \cite{Kock}, with further contributions by Z\" oberlein \cite{Zoberlein}, and the term KZ-pseudomonad is therefore also used.  A beautiful insight was later provided by Street \cite{Street} who observed that these structures are encoded by the algebraists' simplicial category when it is extended to a $2$-category using the pointwise ordering.


Since pseudomonads are already quite complicated, and lax-idempotent pseudomonads are pseudomonads satisfying further properties, one could be forgiven for thinking that they are very complicated structures.  However, the second main perspective, due to Marmolejo and Wood \cite{MW}, shows that they are in fact less complicated than pseudomonads: indeed, a lax-idempotent pseudomonad $T$ is completely specified by objects $TX$ and morphisms $\eta_X \colon X \to TX$ satisfying a couple of properties concerning the existence and preservation of left Kan extensions.  This much shorter but equivalent definition, called a left Kan pseudomonad, makes working with lax-idempotent pseudomonads quite easy in practise, and fits naturally with the idea that they describe cocompletion processes.

The present paper arises from the following questions.  Given a pseudofunctor $U \colon \ca \to \cb$, known to have a left biadjoint $F$, how do we know (1) whether $U$ is monadic in the appropriate bicategorical sense and (2) whether the induced pseudomonad $T=UF$ is lax-idempotent?  The first question is answered by Beck's theorem for pseudomonads \cite{CMV} which is the natural extension of Beck's classic monadicity theorem \cite{Beck} to two dimensions.  It is a useful result because it involves only checking properties of $U$, and does not require one to have an explicit description of $F$ or $T$.  

The main goal of the present paper is therefore to answer the second question: that is, to provide an easily verifiable condition on $U$ ensuring that the induced pseudomonad $T$ is lax-idempotent.  This is achieved in Theorem \ref{thm:laxidempotent} and the key criterion is what we call the \emph{colax bilimit property}.  It involves the lifting of certain limits along $U$ and is easily verified in all the examples we know. Furthermore, as explained in Example \ref{classic}, this criterion is inspired by, and generalises, an important construction in two-dimensional monad theory. 

In forthcoming joint work with Jel\'{i}nek \cite{BJ}, we will apply Theorem \ref{thm:laxidempotent} and the later results developed here to show that the abstract type theories of Uemura \cite{Uemura} are monadic for a colax-idempotent pseudomonad on a suitable $2$-category of signatures.

%

Let us now give an overview of the paper.  In Section \ref{sect:prelim} we quickly recall a few $2$-categorical concepts and facts that we will need in the paper.  In Section \ref{sect:core} we prove our core result, Theorem \ref{thm:laxidempotent}, which concerns when a biadjunction gives rise to a lax-idempotent pseudomonad.  In Section \ref{sect:char} we build on this result to characterise lax-idempotent pseudomonads on suitably complete $2$-categories --- see Theorem \ref{thm:char} --- and use it to give a refinement of Beck's theorem for pseudomonads, in the lax-idempotent setting, in Theorem \ref{thm:monadicity2}.

\subsection{Acknowledgements}

I became aware of the need for a recognition result for lax-idempotent pseudomonads through my collaboration with V\' it Jel\' inek investigating bicolimit presentations and monadicity for type theories.  I am grateful to V\' it for valuable discussions on this topic and also for feedback on an early draft of this paper, which led to the removal of a redundant hypothesis in Theorem \ref{thm:laxidempotent}.  

\section{Preliminaries}\label{sect:prelim}

We begin by briefly recalling a few $2$-categorical concepts and facts that we will use in what follows.

\subsection{Basic $2$-categorical concepts and notation}
In this paper we will work with (strict) $2$-categories denoted by $\ca,\cb,\cc$ and so on.  Between $2$-categories $\ca$ and $\cb$ is the $2$-category $Ps(\ca,\cb)$ of pseudofunctors, pseudonatural transformations and modifications.  Likewise, there are the $2$-categories of pseudofunctors and lax (respectively colax) natural transformations $Lax(\ca,\cb)$ and $Colax(\ca,\cb)$.

A \emph{biadjunction} $F \dashv U \colon \ca \to \cb$ consists of pseudofunctors $F$ and $U$ together with equivalences $\phi_{A,B} \colon \ca(FA,B) \to \cb(A,UB)$ pseudonatural in $A$ and $B$ \cite{Street}.  As with ordinary adjunctions, biadjunctions also admit a formulation in terms of units and counits.  In these terms, one has pseudonatural transformations $\epsilon \colon FU \to 1$ and $\eta \colon 1 \to UF$ together with invertible modifications $\theta \colon U\epsilon \circ \eta_U \cong 1_U$ and $\phi \colon \epsilon_F \circ F\eta \cong 1_F$ satisfying the swallowtail identities --- see, for instance, Section 1.7 of \cite{Gray}.  

The biadjunction is said to be \emph{lax-idempotent} if the invertible modification $\theta$ is the counit of an adjunction $U\epsilon \dashv \eta_U \in Ps(\ca,\cb)$.  

The theory of pseudomonads has been developed by multiple authors across a range of papers including \cite{Street, Marmolejo} and \cite{CMV}.  The recent thesis of \v{S}t\v{e}p\'{a}n \cite{Stepan} (sections 2.3-2.5) provides a useful overview and is what we will primarily follow.  A \emph{pseudomonad} involves a pseudofunctor $T \colon \cc \to \cc$, pseudonatural transformations $\mu \colon T^2 \to T$ and $\eta \colon 1 \to T$, together with invertible modifications $\alpha \colon \mu \circ T\mu \cong \mu \circ \mu_T$, $\lambda \colon \mu \circ \eta_T \cong 1_T$ and $\rho \colon 1_T \cong \mu \circ T\eta$ satisfying higher associativity and unit equations --- see Section 2.4 of \cite{Stepan}.  

For $T$ a pseudomonad on $\cc$, we let $\PTAlg_p$ denote its $2$-category of pseudoalgebras, pseudomorphisms and algebra $2$-cells.  This comes equipped with a forgetful $2$-functor $U^T \colon \PTAlg_p \to \cc$ which has a left biadjoint $F^T$ satisfying $T = U^T \circ F^T$.

The pseudomonad $T$ is said to be \emph{lax-idempotent} if the invertible modification $\lambda$ forms the counit of an adjunction $\mu \dashv \eta_T$ in $Ps(\cc,\cc)$.   Whilst this formulation of lax idempotency is the one we will use, there are several other formulations and we refer the reader to Section 2.5 of \cite{Stepan} for an overview of them.  By Proposition 2.5.12 of \cite{Stepan}, a biadjunction $F \dashv U$ is lax-idempotent if and only if the associated pseudomonad $T=UF$ is lax-idempotent.

\subsection{``Adjoint adjunctions'' and adjunctions with invertible counit}

It is well known that each equivalence in a $2$-category can be upgraded to an adjoint equivalence.  A useful tool for us is the less well known generalisation of this result to the context of adjunctions, which we now recall.

\begin{lemma}[Lemma 2.1.11 of \cite{RV}] \label{lemma:adjunction} 
Let $f \colon A \to B$ and $u\colon B \to A$ be morphisms in a $2$-category $\cc$. Then $f \dashv u$ if and only if there exists 2-cells $\eta \colon 1 \to uf$ and $\epsilon \colon fu \to 1$ such that the composites $u\epsilon \circ \eta_u \colon u \to  ufu \to u$ and $\epsilon_f \circ f\eta \colon f \to fuf \to f$ are invertible.  (Moreover, the counit of the adjunction can be taken to be $\epsilon$.)
\end{lemma}
%
%
%

Because lax-idempotent pseudomonads involve adjunctions $f \dashv u$ with invertible counit, we now give a variant of the above result starting with an invertible $2$-cell $\epsilon \colon fu \cong 1$.  In fact to check that $f \dashv u$ in this setting, the $2$-cell $\eta$ does not need to have any relationship to $\epsilon$.  Applying this will make our later calculations much easier.

\begin{corollary}\label{cor:lali}
Let $f \colon A \to B$ and $u\colon B \to A$ be morphisms in a $2$-category $\cc$ and suppose that we have an invertible $2$-cell $\epsilon \colon fu \cong 1$.  Then there is an adjunction $f \dashv u$ with counit $\epsilon$ if and only if there exists a $2$-cell $\eta \colon 1 \to uf$ such that the composites $f\eta$ and $\eta_u$ are invertible.
\end{corollary}
\begin{proof}
Since $\epsilon$ is invertible, the two triangular composites of Lemma \ref{lemma:adjunction} are invertible if and only if $f\eta$ and $\eta_u$ are invertible.
\end{proof}

\subsection{Colax bilimits of arrows}

The main technical condition in Theorem~\ref{thm:laxidempotent} concerns colax bilimits of arrows.\begin{footnote}{The terminology \emph{oplax bilimit} is also common.  We prefer \emph{colax bilimit} here as it is captures the fact that colax and lax structures are dual under the $(-)^{co}$-duality, which reverses $2$-cells.  This terminology also interacts better with our other main concept, lax-idempotent pseudomonad, whose dual is typically called colax-idempotent.}\end{footnote}  Given a morphism $f\colon A \to B$ in a $2$-category $\cc$, its colax bilimit is specified by an object $C$ together with $1$-cells and a $2$-cell as below
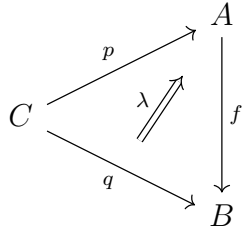
\[\begin{tikzcd}
	&& A \\
	C \\
	&& B
	\arrow["f", from=1-3, to=3-3]
	\arrow["p", from=2-1, to=1-3]
	\arrow[""{name=0, anchor=center, inner sep=0}, "q"', from=2-1, to=3-3]
	\arrow["\lambda"{pos=0.4}, between={0.2}{0.7}, Rightarrow, from=0, to=1-3]
\end{tikzcd}\]
such that for all $X \in \cc$, the induced functor $$\cc(X,C) \to \cc(X,B)/\cc(X,f) \colon k \mapsto (qk,\lambda k \colon qk \to fpk,pk)$$ to the comma category is an equivalence of categories.  

In elementary terms, this condition involves a $1$-dimensional and $2$-dimensional universal property.  The $1$-dimensional universal property says that given a further diagram as below left, there exists a $1$-cell $w \colon X \to C$ and invertible $2$-cells as depicted on the right below
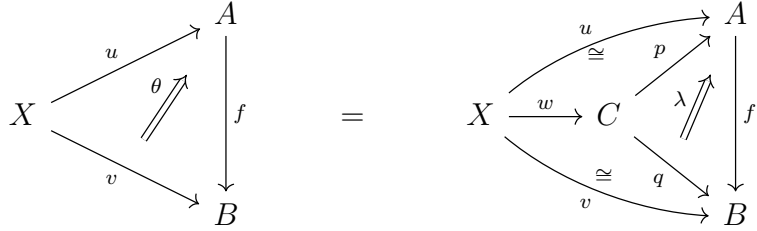
\[\begin{tikzcd}
	&& A &&&& A \\
	X &&& {=} & X & C \\
	&& B &&&& B
	\arrow["f", from=1-3, to=3-3]
	\arrow["f", from=1-7, to=3-7]
	\arrow["u", from=2-1, to=1-3]
	\arrow[""{name=0, anchor=center, inner sep=0}, "v"', from=2-1, to=3-3]
	\arrow[""{name=1, anchor=center, inner sep=0}, "u", curve={height=-12pt}, from=2-5, to=1-7]
	\arrow["w", from=2-5, to=2-6]
	\arrow[""{name=2, anchor=center, inner sep=0}, "v"', curve={height=12pt}, from=2-5, to=3-7]
	\arrow["p", from=2-6, to=1-7]
	\arrow[""{name=3, anchor=center, inner sep=0}, "q"', from=2-6, to=3-7]
	\arrow["\theta", between={0.2}{0.7}, Rightarrow, from=0, to=1-3]
	\arrow["\cong", shift right=3, draw=none, from=1, to=2-6]
	\arrow["\lambda"{pos=0.4}, between={0.2}{0.7}, Rightarrow, from=3, to=1-7]
	\arrow["\cong", shift right=2, draw=none, from=2-6, to=2]
\end{tikzcd}\]
such that the above equality holds.  The $2$-dimensional universal property can be likewise extracted from the definition and since we won't need to make use of it in Theorem \ref{thm:laxidempotent}, we leave it to the reader.

\section{Recognising lax-idempotent biadjunctions and pseudomonads}\label{sect:core}

The following result is the core technical result of the paper.  It describes conditions under which the pseudomonad induced by a biadjunction is lax-idempotent.

\begin{theorem}\label{thm:laxidempotent}
Consider a biadjunction $F \dashv U \colon \ca \to \cb$ and suppose that $U \colon \ca \to \cb$ is locally full.  Then the biadjunction is lax-idempotent if
\begin{itemize}
\item \emph{Colax bilimit property}: each morphism $f \colon UA \to UB$ has a colax bilimit of the form
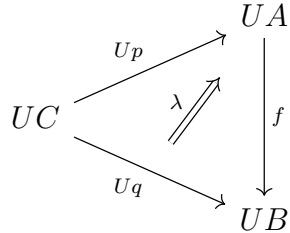
\[\begin{tikzcd}
	&& UA \\
	UC \\
	&& UB
	\arrow["f", from=1-3, to=3-3]
	\arrow["Up", from=2-1, to=1-3]
	\arrow[""{name=0, anchor=center, inner sep=0}, "Uq"', from=2-1, to=3-3]
	\arrow["\lambda"{pos=0.4}, between={0.2}{0.7}, Rightarrow, from=0, to=1-3]
\end{tikzcd}\]
\end{itemize}
In particular, the induced pseudomonad $T=UF$ is then lax-idempotent.
\end{theorem}

\begin{proof}
Let $\eta \colon 1 \Rightarrow UF$ and $\epsilon \colon FU \Rightarrow 1$ denote the pseudonatural transformations forming the unit and counit of the biadjunction.  
We must prove that $U\epsilon \dashv \eta_{U}$ with counit $\theta \colon U\epsilon \circ \eta_{U} \cong 1_{U}$.

  By doctrinal adjunction \cite{Kelly} this will be the case if and only if it is so pointwise -- that is, just when $U\epsilon_{A} \dashv \eta_{UA}$ with counit $\theta_A$.  Therefore, by Corollary \ref{cor:lali}, it suffices to find a $2$-cell $1_{UFUA} \Rightarrow \eta_{UA} \circ U\epsilon_{A}$ which becomes invertible both when post-composed by $U{\epsilon}_{A}$ and when pre-composed by $\eta_{UA}$.

To this end, consider the colax bilimit below.

\[\begin{tikzcd}
	&& UA \\
	UC \\
	&& UFUA
	\arrow["{\eta_{UA}}", from=1-3, to=3-3]
	\arrow["Up", from=2-1, to=1-3]
	\arrow[""{name=0, anchor=center, inner sep=0}, "Uq"', from=2-1, to=3-3]
	\arrow["\lambda"{pos=0.4}, between={0.2}{0.7}, Rightarrow, from=0, to=1-3]
\end{tikzcd}\]

By its universal property, we obtain a morphism $h \colon UA \to UC$ together with invertible $2$-cells as depicted below such that the two composite $2$-cells are equal.  (We will not, in fact, need to name these $2$-cells.)

\begin{equation}\label{eq:firstdiagram}
\begin{tikzcd}
	& UA &&&& UA \\
	UA && {=} & UA & UC \\
	& UFUA &&&& UFUA
	\arrow[""{name=0, anchor=center, inner sep=0}, "{\eta_{UA}}", from=1-2, to=3-2]
	\arrow["{\eta_{UA}}", from=1-6, to=3-6]
	\arrow["1", from=2-1, to=1-2]
	\arrow["{\eta_{UA}}"', from=2-1, to=3-2]
	\arrow[""{name=1, anchor=center, inner sep=0}, "1", curve={height=-12pt}, from=2-4, to=1-6]
	\arrow["h", from=2-4, to=2-5]
	\arrow[""{name=2, anchor=center, inner sep=0}, "{\eta_{UA}}"', curve={height=12pt}, from=2-4, to=3-6]
	\arrow["Up", from=2-5, to=1-6]
	\arrow[""{name=3, anchor=center, inner sep=0}, "Uq"', from=2-5, to=3-6]
	\arrow["{=}"{marking, allow upside down, pos=0.6}, draw=none, from=2-1, to=0]
	\arrow["\cong", shift right=3, draw=none, from=1, to=2-5]
	\arrow["\lambda"{pos=0.4}, between={0.2}{0.7}, Rightarrow, from=3, to=1-6]
	\arrow["\cong", shift right=2, draw=none, from=2-5, to=2]
\end{tikzcd}
\end{equation}

By the universal property of $FUA$, there exists a $1$-cell and invertible $2$-cell as below.

\begin{equation}\label{eq:seconddiagram}
\begin{tikzcd}
	UA && UFUA && UC
	\arrow["{\eta_{UA}}", from=1-1, to=1-3]
	\arrow[""{name=0, anchor=center, inner sep=0}, "h", curve={height=-30pt}, from=1-1, to=1-5]
	\arrow["{Uh'}", from=1-3, to=1-5]
	\arrow["\cong"{pos=0.7}, shift right=3, draw=none, from=0, to=1-3]
\end{tikzcd}
\end{equation}

Now consider the composite $2$-cell.  

\[\begin{tikzcd}
	&& UA \\
	UFUA & UC \\
	&& UFUA
	\arrow["{\eta_{UA}}", from=1-3, to=3-3]
	\arrow["{Uh'}", from=2-1, to=2-2]
	\arrow["Up", from=2-2, to=1-3]
	\arrow[""{name=0, anchor=center, inner sep=0}, "Uq"', from=2-2, to=3-3]
	\arrow["\lambda"{pos=0.4}, between={0.2}{0.7}, Rightarrow, from=0, to=1-3]
\end{tikzcd}\]

We claim that (1) $Up \circ Uh' \cong U\epsilon_{A}$ and (2) $Uq \circ Uh' \cong 1_{UFUA}$.  By the bicategorical universal property of $FUA$, this is equally to show that both sides become isomorphic under pre-composition by $\eta_{UA}$.

Now for (1) we have $$Up \circ Uh' \circ \eta_{UA} \cong Up \circ h \cong 1_{UA} \cong U\epsilon_{A} \circ \eta_{UA}$$ 
Here, the first isomorphism is by definition of $h'$; the second by the definition of $h$, and the third by the triangle isomorphism for the biadjunction.  Having proven (1), we turn to (2).  We have $$Uq \circ Uh' \circ \eta_{UA} \cong Uq \circ h \cong \eta_{UA}$$ again by definition of $h'$ and $h$.  Combining (1) and (2) with the above, we hence obtain a $2$-cell

\begin{equation}\label{eq:thirddiagram}
\begin{tikzcd}
	&& UA \\
	UFUA & UC \\
	&& UFUA
	\arrow["{\eta_{UA}}", from=1-3, to=3-3]
	\arrow[""{name=0, anchor=center, inner sep=0}, "{U\epsilon_{A}}", curve={height=-24pt}, from=2-1, to=1-3]
	\arrow["{Uh'}", from=2-1, to=2-2]
	\arrow[""{name=1, anchor=center, inner sep=0}, "1"', curve={height=24pt}, from=2-1, to=3-3]
	\arrow["Up", from=2-2, to=1-3]
	\arrow[""{name=2, anchor=center, inner sep=0}, "Uq"', from=2-2, to=3-3]
	\arrow["\cong", draw=none, from=0, to=2-2]
	\arrow["\cong"', draw=none, from=1, to=2-2]
	\arrow["\lambda"{pos=0.4}, between={0.2}{0.7}, Rightarrow, from=2, to=1-3]
\end{tikzcd}
\end{equation}

and it remains to prove that this $2$-cell becomes invertible both when post-composed by $U{\epsilon}_{A}$ and when pre-composed by $\eta_{UA}$.  Under pre-composition we obtain

\begin{equation}\label{eq:fourthdiagram}
\begin{tikzcd}
	&&& UA \\
	UA & UFUA & UC \\
	&&& UFUA
	\arrow["{\eta_{UA}}", from=1-4, to=3-4]
	\arrow["{\eta_{UA}}", from=2-1, to=2-2]
	\arrow[""{name=0, anchor=center, inner sep=0}, "{U\epsilon_{A}}", curve={height=-24pt}, from=2-2, to=1-4]
	\arrow["{Uh'}", from=2-2, to=2-3]
	\arrow[""{name=1, anchor=center, inner sep=0}, "1"', curve={height=24pt}, from=2-2, to=3-4]
	\arrow["Up", from=2-3, to=1-4]
	\arrow[""{name=2, anchor=center, inner sep=0}, "Uq"', from=2-3, to=3-4]
	\arrow["\cong", draw=none, from=0, to=2-3]
	\arrow["\cong"', draw=none, from=1, to=2-3]
	\arrow["\lambda"{pos=0.4}, between={0.2}{0.7}, Rightarrow, from=2, to=1-4]
\end{tikzcd}
\end{equation}

which will be invertible just when

\[\begin{tikzcd}
	&&& UA \\
	UA & UFUA & UC \\
	&&& UFUA
	\arrow["{\eta_{UA}}", from=1-4, to=3-4]
	\arrow["{\eta_{UA}}", from=2-1, to=2-2]
	\arrow["{Uh'}", from=2-2, to=2-3]
	\arrow["Up", from=2-3, to=1-4]
	\arrow[""{name=0, anchor=center, inner sep=0}, "Uq"', from=2-3, to=3-4]
	\arrow["\lambda"{pos=0.4}, between={0.2}{0.7}, Rightarrow, from=0, to=1-4]
\end{tikzcd}\]

is invertible.  By  \eqref{eq:seconddiagram}, this is equally to ask that
\[\begin{tikzcd}
	&&& UA \\
	UA && UC \\
	&&& UFUA
	\arrow["{\eta_{UA}}", from=1-4, to=3-4]
	\arrow["h", from=2-1, to=2-3]
	\arrow["Up", from=2-3, to=1-4]
	\arrow[""{name=0, anchor=center, inner sep=0}, "Uq"', from=2-3, to=3-4]
	\arrow["\lambda"{pos=0.4}, between={0.2}{0.7}, Rightarrow, from=0, to=1-4]
\end{tikzcd}\]
is invertible, or equivalently that 

\[\begin{tikzcd}
	&&& UA \\
	UA && UC \\
	&&& UFUA
	\arrow["{\eta_{UA}}", from=1-4, to=3-4]
	\arrow[""{name=0, anchor=center, inner sep=0}, "1", curve={height=-24pt}, from=2-1, to=1-4]
	\arrow["h", from=2-1, to=2-3]
	\arrow[""{name=1, anchor=center, inner sep=0}, "{\eta_{UA}}"', curve={height=24pt}, from=2-1, to=3-4]
	\arrow["Up", from=2-3, to=1-4]
	\arrow[""{name=2, anchor=center, inner sep=0}, "Uq"', from=2-3, to=3-4]
	\arrow["\cong", draw=none, from=0, to=2-3]
	\arrow["\cong"', draw=none, from=1, to=2-3]
	\arrow["\lambda"{pos=0.4}, between={0.2}{0.7}, Rightarrow, from=2, to=1-4]
\end{tikzcd}\]

 is invertible, where we paste by the invertible $2$-cells in Equation \eqref{eq:firstdiagram}.  By this equation, the above composite the identity $2$-cell, proving the claim.

We must also prove that

\[\begin{tikzcd}
	&& UA \\
	UFUA & UC \\
	&& UFUA && UA
	\arrow["{\eta_{UA}}", from=1-3, to=3-3]
	\arrow[""{name=0, anchor=center, inner sep=0}, "{U\epsilon_{A}}", curve={height=-24pt}, from=2-1, to=1-3]
	\arrow["{Uh'}", from=2-1, to=2-2]
	\arrow[""{name=1, anchor=center, inner sep=0}, "1"', curve={height=24pt}, from=2-1, to=3-3]
	\arrow["Up", from=2-2, to=1-3]
	\arrow[""{name=2, anchor=center, inner sep=0}, "Uq"', from=2-2, to=3-3]
	\arrow["{U\epsilon_{A}}"', from=3-3, to=3-5]
	\arrow["\cong", draw=none, from=0, to=2-2]
	\arrow["\cong"', draw=none, from=1, to=2-2]
	\arrow["\lambda"{pos=0.4}, between={0.2}{0.7}, Rightarrow, from=2, to=1-3]
\end{tikzcd}\]

is invertible or, equivalently, that

\[\begin{tikzcd}
	&& UA \\
	UFUA & UC \\
	&& UFUA && UA
	\arrow["{\eta_{UA}}", from=1-3, to=3-3]
	\arrow[""{name=0, anchor=center, inner sep=0}, "1", from=1-3, to=3-5]
	\arrow[""{name=1, anchor=center, inner sep=0}, "{U\epsilon_{A}}", curve={height=-24pt}, from=2-1, to=1-3]
	\arrow["{Uh'}", from=2-1, to=2-2]
	\arrow[""{name=2, anchor=center, inner sep=0}, "1"', curve={height=24pt}, from=2-1, to=3-3]
	\arrow["Up", from=2-2, to=1-3]
	\arrow[""{name=3, anchor=center, inner sep=0}, "Uq"', from=2-2, to=3-3]
	\arrow["{U\epsilon_{A}}"', from=3-3, to=3-5]
	\arrow["\cong", draw=none, from=1, to=2-2]
	\arrow["\cong"', draw=none, from=2, to=2-2]
	\arrow["\lambda"{pos=0.4}, between={0.2}{0.7}, Rightarrow, from=3, to=1-3]
	\arrow["\cong"{description, pos=0.6}, shift right, draw=none, from=3-3, to=0]
\end{tikzcd}\]

is invertible, where we paste by the invertible $2$-cell for the biadjunction.  But, by local fullness of $U$, this is of the form $U\gamma$ for some $\gamma \colon \epsilon_{A} \Rightarrow \epsilon_{A}$.  Therefore, to show that $\gamma$ is invertible (and in particular $U\gamma$) is equivalently to show that it is inverted by the hom-equivalence $$U- \circ \eta_{UA} \colon \ca(FUA,A) \to \cb(UA,UA)$$ of the biadjunction; in other words that $U\gamma$ becomes invertible under precomposition by $\eta_{UA}$.  We already proved this above, in showing that \eqref{eq:fourthdiagram} is invertible.

\end{proof}

\begin{remark}
In the previous result we did not use the 2-dimensional universal property of the colax bilimit, which could be omitted from the assumptions.
\end{remark}

\begin{remark}\label{rem:dual}[\emph{Dual and pseudo versions}]
Theorem \ref{thm:laxidempotent} has a dual, obtained by applying the duality on a $2$-category $\cc$ which reverses the direction of $2$-cells.  This interchanges lax and colax-idempotent pseudomonads and lax and colax bilimits of arrows.  

It also has a \emph{pseudo} version.  From Section 2.5 of \cite{Stepan}, a lax-idempotent biadjunction is pseudo-idempotent if the adjunction $U\epsilon \dashv \eta U$ is an adjoint equivalence whilst a lax-idempotent pseudomonad $T$ is pseudo-idempotent if the adjunction $\mu \dashv \eta T$ is actually an adjoint equivalence.  Moreover the \emph{bilimit} of an arrow is like the colax bilimit, except that the colax cone in its definition is a pseudo-cone --- that is, it has invertible $2$-cell $\lambda$ --- and its universal property holds only with respect to other pseudo-cones.  In the pseudo version of Theorem~\ref{thm:laxidempotent}, lax-idempotent biadjunctions and pseudomonads are replaced by pseudo-idempotent ones and colax bilimits by bilimits.  The proof of the corresponding result is a trivial adaptation. \end{remark}

\begin{examples} [\emph{Categories with structure}]
Let us now describe a few examples which illustrate how the colax bilimit property in Theorem \ref{thm:laxidempotent} is easily verified in practise.

Firstly, observe that in $\Cat$, the colax bilimit of a functor $f \colon A \to B$ is the comma category $B/f$ whose objects are triples $(b,\alpha \colon b \to fa,a)$.  The projections $p \colon B/f \to A$ and $q \colon B/f \to B$ are the evident functors sending $(b,\alpha \colon b \to fa,a)$ to $a$ and $b$ respectively, whilst the natural transformation $\lambda \colon q \to f \circ p$ sends the above triple to $\alpha$.
\begin{enumerate}
\item Let $\BCop$ be the $2$-category of categories admitting binary coproducts, binary coproduct preserving functors and natural transformations, and consider the evident forgetful $2$-functor $U \colon \BCop \to \Cat$.  If $A$ and $B$ as above have binary coproducts, we must show that $B/f$ has binary coproducts and the projections $p$ and $q$ preserve them --- in other words, that binary coproducts are \emph{pointwise} in $B/f$.  Indeed, given $(b,\alpha \colon b \to fa,a), (b',\alpha \colon b' \to fa',a') \in B/f$, it is an easy exercise to show that the coproduct is given by 
\[\begin{tikzcd}
	{(b+b', b+b'} & {fa +fa'} & {f(a+a'),a+a')}
	\arrow["{\alpha+\alpha'}", from=1-1, to=1-2]
	\arrow["can", from=1-2, to=1-3]
\end{tikzcd}\]
where $can$ is the canonical comparison between the coproducts.  Thus $U \colon \BCop \to \Cat$ satisfies the colax bilimit property and so, by Theorem~\ref{thm:laxidempotent}, induces a lax-idempotent pseudomonad.  

An essentially identical argument applies to forgetful $2$-functors $U \colon \PhiCol \to \Cat$ from $2$-categories of categories equipped with colimits of a given class $\Phi$ and the functors preserving them.
\item On the other hand, limits in $B/f$ are slightly different.  For instance, if $A$ and $B$ have binary products, the product of  $(b,\alpha \colon b \to fa,a)$ and $(b',\alpha \colon b' \to fa',a') \in B/f$ exists pointwise so long as $f$ preserves the product $a \times a'$, as below
\[\begin{tikzcd}
	{(b\times b', b\times b'} & {fa \times fa'} & {f(a\times a' ),a\times a')}
	\arrow["{\alpha\times \alpha'}", from=1-1, to=1-2]
	\arrow["can^{-1}", from=1-2, to=1-3]
\end{tikzcd}\]
since defining the structure map requires the canonical comparison $f(a\times a' ) \to fa \times fa'$ between products to be invertible.  This illustrates the fact that forgetful $2$-functors $U \colon \PhiLim \to \Cat$ from $2$-categories of categories equipped with limits of a given class $\Phi$ \emph{preserve} colax bilimits of arrows but do \emph{not} satisfy the colax bilimit property.
\item Let $\Lex$ denote the sub-$2$-category of $\Cat$ containing categories with finite limits and functors preserving them.  By the above, the inclusion $ \Lex \to \Cat$ is closed under colax bilimits of arrows.
Now let $\Reg$ denote the $2$-category of regular categories and regular functors and $U \colon \Reg \to \Lex$ the evident inclusion.  

In addition to finite limits, a regular category involves coequalisers of kernel pairs which are required to be pullback stable.  Now if $f \colon A \to B$ is a finite limit preserving functor between regular categories, then $B/f$ admits finite limits by (2) above and coequalisers of kernel pairs by (1) above, both of which are pointwise as in $\ca$ and $\cb$.  Since $p \colon B/f \to A$ and $q \colon B/f \to B$ are jointly conservative and preserve these limits and colimits, they also reflect them --- as such, the pullback stability property also holds in $B/f$, since it does so in $A$ and $B$.  In this way, we see that $U \colon \Reg \to \Lex$ satisfies the colax bilimit property.   Similar remarks apply to all of the usual doctrines such as Barr-exact, coherent categories and pretopoi.
\end{enumerate}

In the above examples, the forgetful $2$-functors in question may be shown to have left biadjoints in a number of ways, and are clearly locally full.  Therefore Theorem \ref{thm:laxidempotent} ensures that the induced pseudomonads are lax-idempotent.  We emphasise though, that these examples are already well known to be describable using lax-idempotent pseudomonads using direct descriptions of the cocompletion pseudomonad in question --- see \cite{Kock} and \cite{GL}.

An application where Theorem \ref{thm:laxidempotent} is needed will appear in our forthcoming paper with Jel\'{i}nek \cite{BJ}, where we will prove that the abstract type theories of Uemura \cite{Uemura} --- certain structured categories called \emph{representable map categories} --- are monadic for a colax-idempotent pseudomonad.  In that case, our proof will be heavily dependent on Theorem \ref{thm:laxidempotent} since we do not have a direct description of the pseudomonad in question.

\end{examples}

\begin{example}\label{classic}[\emph{Two-dimensional monad theory}]
Let us now explain how Theorem \ref{thm:laxidempotent} generalises a classical result from $2$-dimensional monad theory (see \cite{BKP, LS}).  

Firstly, recall that for $T$ a $2$-monad, one has the $2$-categories $\TAlg_s$, $\TAlg_p$ and $\TAlg_l$ of strict algebras, strict, pseudo and lax morphisms respectively.  There are identity-on-objects $2$-functors $\iota_p \colon \TAlg_s \hookrightarrow \TAlg_p$ and $\iota_l \colon \TAlg_s \hookrightarrow \TAlg_l$ and both have left $2$-adjoints if $\TAlg_s$ is sufficiently cocomplete (in particular, if $T$ is a $2$-monad with rank on a complete and cocomplete $2$-category --- see 3.13 of \cite{BKP}).  The left adjoints $Q_p$ and $Q_l$ are known as the pseudomorphism classifier and lax morphism classifier respectively.  

We write $\epsilon_A  \colon Q_p A \to A$ and $\eta_A \colon A \rightsquigarrow Q_p A$ for the counit and unit of the first $2$-adjunction, for which one triangle equation gives $\epsilon_A \circ \eta_A = 1$.  Lemma 4.2 of \cite{BKP} proves that there is an adjoint equivalence $\epsilon_A \dashv \eta_A$ with identity counit, which is equally to say that the $2$-adjunction is pseudo-idempotent.  The proof uses pseudolimits of arrows in $\TAlg_p$ and amounts to verifying that $\iota_p \colon \TAlg_s \hookrightarrow \TAlg_p$ satisfies the bilimit property, as in the pseudo-version in Remark \ref{rem:dual}.  Note that pseudo-idempotence in this context is important, because it implies the bi-cocompleteness of $\TAlg_p$ --- see Theorem 5.8 of \cite{BKP}.  Lemma 2.5 of \cite{LS} extended this argument to prove that the $2$-adjunction $Q_l \dashv \iota_l$ is lax-idempotent and their proof essentially amounts to verifying the colax bilimit property for $\iota_l \colon \TAlg_s \hookrightarrow \TAlg_l$.


Therefore Theorem \ref{thm:laxidempotent} and its pseudo-variant are generalisations of results from $2$-dimensional monad theory, which apply to general pseudofunctors as opposed to the particular $2$-functors $\iota_p \colon \TAlg_s \hookrightarrow \TAlg_p$ and $\iota_l \colon \TAlg_s \hookrightarrow \TAlg_l$, and which allow biadjoints, pseudomonads and bilimits rather than strict $2$-adjoints and $2$-limits.

\end{example}

\section{A characterisation of lax-idempotent pseudomonads}\label{sect:char}

The main result of this section, Theorem \ref{thm:char}, asserts that if $T$ is a pseudomonad on $\cc$ admitting colax bilimits of arrows, then $T$ is lax idempotent if and only if $U^T \colon \PTAlg_p \to \cc$ is locally full and satisfies the colax bilimit property.

The main task is to construct colax bilimits of arrows in $\PTAlg_p$ assuming they exist in $\cc$.  We could do this by adapting the techniques of Theorem 3.2 of \cite{Lack} for strict $2$-monads and their algebras but the proof would be complicated.  However, because we are dealing with lax-idempotent pseudomonads, there is a simpler argument.  It uses the \emph{Kan injectivity} approach to lax-idempotent pseudomonads developed in \cite{DLS}.

\subsection{Kan injectives and lax-idempotent pseudomonads}

Let $\mathcal J$ be a class of morphisms in a $2$-category $\cc$.     

\begin{itemize}
\item A (left) Kan injective is an object $X \in \cc$ such that for each $\alpha \colon A \to B \in \cc$, the functor $\cc(\alpha,X) \colon \cc(B,X) \to \cc(A,X)$ has a left adjoint $lan^X_\alpha$ with invertible unit.  (In other words, $X$ admits all left Kan extensions of morphisms $f \colon A \to X$ along $\alpha$ and the units of such extensions are invertible.)
\item A lax morphism of Kan injectives $f \colon X \to Y$ is simply a morphism in $\cc$.  It is a pseudomorphism (in \cite{DLS} a homomorphism) if it preserves left Kan extensions along morphisms in $\mathcal J$.  In terms of adjunctions, this says that the commutative square below left
\[\begin{tikzcd}
	{\cc(B,X)} && {\cc(B,Y)} && {\cc(B,X)} && {\cc(B,Y)} \\
	{\cc(A,X)} && {\cc(A,Y)} && {\cc(A,X)} && {\cc(A,Y)}
	\arrow[""{name=0, anchor=center, inner sep=0}, "{\cc(B,f)}", from=1-1, to=1-3]
	\arrow["{\cc(\alpha,X)}"', from=1-1, to=2-1]
	\arrow["{\cc(\alpha,Y)}", from=1-3, to=2-3]
	\arrow[""{name=1, anchor=center, inner sep=0}, "{\cc(B,f)}", from=1-5, to=1-7]
	\arrow[""{name=2, anchor=center, inner sep=0}, "{\cc(A,f)}"', from=2-1, to=2-3]
	\arrow["{lan^X_\alpha}", from=2-5, to=1-5]
	\arrow[""{name=3, anchor=center, inner sep=0}, "{\cc(A,f)}"', from=2-5, to=2-7]
	\arrow["{lan^Y_\alpha}"', from=2-7, to=1-7]
	\arrow["{=}", shift right=3, draw=none, from=0, to=2]
	\arrow[between={0.3}{0.7}, Rightarrow, from=3, to=1]
\end{tikzcd}\]
satisfies the Beck-Chevalley condition for each $\alpha \in \mathcal J$: that is to say, its mate (depicted above right) through the adjunctions $lan^X_\alpha \dashv \cc(\alpha,X)$ and $lan^Y_\alpha \dashv \cc(\alpha,Y)$ is invertible.
\end{itemize} 

We obtain a $2$-category $\Kan(\mathcal J)_l$ of Kan injectives, lax morphisms and all $2$-cells between them in $\cc$, and a locally full sub-$2$-category $\Kan(\mathcal J)_p \hookrightarrow \Kan(\mathcal J)_l$ containing the same objects but only the pseudomorphisms.

Given a lax-idempotent pseudomonad $T$ on $\cc$, let $\mathcal U = \{\eta_X \colon X \to TX : X \in \cc\}$.  This is the associated \emph{left Kan pseudomonad}, in the sense of \cite{MW}.  Indeed the forgetful $2$-functor $U^T \colon \PTAlg_p \to \cc$ takes its values in Kan injectives with respect to $\mathcal U$, and furthermore each Kan injective underlies a pseudo-algebra.  This is captured by the following result, which builds on the results of \cite{MW}.

\begin{theorem}[Theorem 2.7 of \cite{DLS}]\label{thm:semantics}
There is a $2$-equivalence $\PTAlg_p \cong \Kan(\mathcal U)_p$ commuting with the forgetful $2$-functors to the base.
\end{theorem}

\begin{remark}
Although we shall not need it, let us also note that there is also a $2$-equivalence $\PTAlg_l \cong \Kan(\mathcal U)_l$ where on the left we have the $2$-category of pseudoalgebras and lax morphisms.  Indeed, this follows from the above result together with the fact (Theorem 2.5.9 of \cite{Stepan}) that the forgetful $2$-functor $\PTAlg_l \to \cc$ is fully faithful on $1$-cells and $2$-cells whilst the inclusion $\Kan(\mathcal U)_l \hookrightarrow \cc$ is so by definition. 
\end{remark}

We now investigate the lifting of various forms of bilimits to $2$-categories of Kan injectives.  We need to prove that $\Kan(\mathcal J)_l \hookrightarrow \cc$ is closed under any colax bilimits of arrows that $\cc$ has, and, in fact, it will be easier to prove this for general colax bilimits.  In fact this result is a slight modification of  Proposition 1.4 of \cite{DLS}, which shows that $\Kan(\mathcal J)_p \hookrightarrow \cc$ is closed under bilimits \cite{DLS}.  We also treat the dual results for bicolimits, since they may be of independent interest.

\begin{notation}
Before proving the proposition, let us fix some notation.
\begin{enumerate}
\item (\emph{Flavours of bilimit}) Let $W \colon \ca \to \Cat$ be a weight (a $2$-functor with small domain) and $D \colon \ca \to \cc$ a pseudofunctor.  A colax natural transformation $p \colon W \to \mathcal \cc(L,D-)$ exhibits $L$ as the \emph{colax bilimit} of $D$ weighted by $W$ if the induced functor $$\lambda_X \colon \cc(X,L) \to Colax(\ca,\Cat)(W,\cc(X,D-)) \colon f \mapsto \cc(f,D-) \circ p$$ is an equivalence of categories for each $X \in \cc$.   The \emph{bilimit} of $D$ is instead specified by a pseudonatural transformation $p \colon W \to \mathcal \cc(L,D-)$ for which the induced functor $\lambda_X \colon \cc(X,L) \to Ps(\ca,\Cat)(W,\cc(X,D-))$ is an equivalence for each $X$.  

In the case that the equivalence $\lambda_X$ is actually an isomorphism, we speak instead of the colax limit and pseudolimit.  In any of these cases, by the \emph{limit projections}, we mean the components $p_{i,x} \colon L \to Di$ for $i \in \ca$ and $x \in Wi$.  

Dually, there are colax bicolimits and bicolimits, colax colimits and pseudocolimits: for $W$ as above and $D \colon \ca^{op} \to \cc$ these instead involve colax and pseudo natural transformations $p \colon W \to \mathcal \cc(D-,C)$ respectively, and we now refer to the components $p_{i,x} \colon Di \to C$ for $i \in I$ and $x \in Wi$ as \emph{colimit inclusions}.  

\item (\emph{Ralis}) The term \emph{rali} refers to the right adjoint $u$ of an adjunction $f \dashv u$ with invertible unit.  In particular, $X$ is Kan injective just when each $\cc(\alpha,X) \colon \cc(B,X) \to \cc(A,X)$, as above, is a rali.
\end{enumerate}
 \end{notation}

\begin{proposition}\label{prop:kaninjective}
Let $\mathcal J$ be a class of morphisms in a $2$-category $\cc$.  Then 
\begin{enumerate}
\item the full sub-2-category $\Kan(\mathcal J)_l \hookrightarrow \cc$ is closed under any colax bilimits that $\cc$ has.  Furthermore the limit projections are pseudomorphisms of Kan injectives, and jointly detect the property of being pseudo;
\item it is also closed under colax bicolimits that $\cc$ has and which are preserved by homming out of the objects which are the domains and codomains of morphisms in $\mathcal J$.  Furthermore, the colimit inclusions are pseudomorphisms of Kan injectives, and jointly detect the property of being pseudo;
\item the sub-2-category $\Kan(\mathcal J)_p \hookrightarrow \cc$ is closed under any bilimits that $\cc$ has;
\item it is also closed under any bicolimits that $\cc$ has and which are preserved by homming out of the objects which are the domains and codomains of morphisms in $J$.
\end{enumerate}
\end{proposition}
\begin{proof}
\begin{enumerate}
\item
Consider a weight $W: \ca \to \Cat$ and pseudofunctor $D \colon \ca \to \cc$ such that each $Di$ is Kan injective.  We form its colax bilimit $L \in \cc$, which comes equipped with a colax natural transformation $p \colon W \to \mathcal \cc(L,D-)$.  Now, we must prove that for $\alpha \colon A \to B \in \ca$, each $\cc(\alpha,L) \colon \cc(B,L) \to \cc(A,L)$ is a rali.  

Consider the commutative square below.
\[\begin{tikzcd}
	{\cc(B,L)} &&&& {\cc(A,L)} \\
	{Colax(\ca,\Cat)(W,\cc(B,D-))} &&&& {Colax(\ca,\Cat)(W,\cc(A,D-))}
	\arrow["{\cc(\alpha,L)}", from=1-1, to=1-5]
	\arrow["{\lambda_B}"', from=1-1, to=2-1]
	\arrow["{\lambda_A}", from=1-5, to=2-5]
	\arrow["{Colax(\ca,\Cat)(W,\cc(\alpha,D-))}"', from=2-1, to=2-5]
\end{tikzcd}\]
Since the vertical maps are equivalences, the upper horizontal will be a rali just when the lower horizontal is one.  Now since $Colax(\ca,\Cat)(W,-)$, like any $2$-functor, preserves ralis, it suffices to show that $\cc(\alpha,D-) \colon \cc(B,D-) \to \cc(A,D-)$ is a rali in $Colax(\ca,\Cat)$.  Its components are ralis by assumption; therefore, by doctrinal adjunction, the left adjoint components combine to a colax natural transformation $lan^D_\alpha \colon \cc(A,D-) \to \cc(B,D-)$ forming an adjunction $lan^D_\alpha \dashv \cc(\alpha,D-)$ in  $Colax(\ca,\Cat)$, whose unit and counit modifications are componentwise the units and counits of the assumed adjunctions.  In particular, the unit modification is invertible, since each of its components is, and we have a rali as required.

Next, we must prove that each of the limit projections $p_{i,x} \colon L \to Di$ for $x \in Wi$ is a pseudomorphism of Kan injectives --- in other words, that the mate of the outer square below is invertible.
\[\begin{tikzcd}
	{\cc(B,L)} &&&&& {\cc(A,L)} \\
	& {(W,\cc(B,D-))} &&& {(W,\cc(A,D-))} \\
	{\cc(B,Di)} &&&&& {\cc(A,Di)}
	\arrow["{\cc(\alpha,L)}", from=1-1, to=1-6]
	\arrow["{\lambda_B}"', from=1-1, to=2-2]
	\arrow["{\cc(B,p_{i,x})}"{description}, from=1-1, to=3-1]
	\arrow["{\lambda_A}", from=1-6, to=2-5]
	\arrow["{\cc(A,p_{i,x})}"{description}, from=1-6, to=3-6]
	\arrow["{(W,\cc(\alpha,D-))}"', from=2-2, to=2-5]
	\arrow["{ev_{i,x}}"', from=2-2, to=3-1]
	\arrow["{ev_{i,x}}", from=2-5, to=3-6]
	\arrow["{\cc(\alpha,Di)}"', from=3-1, to=3-6]
\end{tikzcd}\]
In the inner components we have abbreviated $Colax(\ca,\Cat)(W,\cc(B,D-))$ by $(W,\cc(B,D-))$ and so on to save space.  Now  the evaluation functor $ev_{i,x}$  sends $q \colon W \to \cc(X,D-))$ to its component $q_{i,x} \colon X \to Di$, where $x \in Wi$ and $i \in \ca$.  In particular, the outer triangles then commute whilst the lower triangle commutes by naturality of evaluation.  Therefore, by functoriality of the mates correspondence (Proposition 2.2 of \cite{Review}), it suffices to prove that the upper and lower quadrilateral have invertible mates.  

The upper quadrilateral does so since $\lambda_B$ and $\lambda_A$ are equivalences.  With regards the lower quadrilateral, observe that by the componentwise construction of the left adjoint to $(W,\cc(\alpha,D-))$, the evaluation functors $ev_{i,x}$ commute with the left adjoints, counit and unit components on the nose --- as such, the mate of the lower quadrilateral, is the identity.  Therefore each $p_{i,x}$ is a morphism of Kan injectives.

For the final part, observe that if $f_i$ are pseudomorphisms of Kan injectives which are jointly conservative, then each composite $ f_i \circ g$ is pseudo just when $g$ is pseudo: indeed, the pseudo property is about $2$-cells being invertible, which is reflected by conservative morphisms.   Accordingly, for the final part of the statement it suffices to show that the functors $\cc(X,p_{i,x})$ are jointly conservative.   As explained above, $\cc(X,p_{i,x})$ is the composite of an equivalence followed by $ev_{i,x}$, and these evaluation functors are jointly conservative since a modification is invertible just when its components are. 

\item For $W$ as before, consider a pseudofunctor $D \colon \ca^{op} \to \cc$ such that each $Di$ is Kan injective.  We suppose that its colax bicolimit $C$ exists with weighted cocone $p \colon W \to \cc(D-,C) \in Colax(\ca,\Cat)$.  For $\alpha \colon A \to B \in \mathcal J$ as before, we must show that $\cc(\alpha,C)$ is a rali.  

Now since $\Cat$ has colax colimits, by their universal property these assemble into a $2$-functor $col^c_W \colon Colax(\ca^{op},\Cat) \to \Cat$.  Now consider the diagram below
\[\begin{tikzcd}
	{col^c_W(\cc(B,D-))} &&& {col^c_W(\cc(A,D-))} \\
	{\cc(B,C)} &&& {\cc(A,C)}
	\arrow[""{name=0, anchor=center, inner sep=0}, "{col^c_W(\cc(\alpha,D-))}", from=1-1, to=1-4]
	\arrow["{can_B}"', from=1-1, to=2-1]
	\arrow["{can_A}", from=1-4, to=2-4]
	\arrow[""{name=1, anchor=center, inner sep=0}, "{\cc(\alpha,C)}"', from=2-1, to=2-4]
	\arrow[shift left=2, draw=none, from=1, to=0]
\end{tikzcd}\]
where the vertical morphisms are the canonical comparisons from the colax colimits, which make the square commute on the nose.  Since $\cc(B,-)$ and $\cc(A,-)$ preserve the colax bicolimit $C$ in $\cc$, these vertical maps are equivalences.  Therefore, we must equally prove that the upper horizontal is a rali.   Just as in the first part of the proposition, $\cc(\alpha,D-)$ is a rali in $Colax(\ca^{op},\Cat)$, and so sent to a rali by $col^c_W$, as required.

Next we show that the colimit inclusions $p_{i,x} \colon Di \to C$ are pseudomorphisms of Kan injectives.  To this end, consider the following diagram
\[\begin{tikzcd}
	{\cc(B,Di)} &&&& {\cc(A,Di)} \\
	& {col_W^c(\cc(B,D-))} && {col_W^c(\cc(A,D-))} \\
	{\cc(B,C)} &&&& {\cc(A,C)}
	\arrow[""{name=0, anchor=center, inner sep=0}, "{\cc(\alpha,Di)}", from=1-1, to=1-5]
	\arrow["{q_{i,x}^B}", from=1-1, to=2-2]
	\arrow["{\cc(B,p_{i,x})}"{description}, from=1-1, to=3-1]
	\arrow["{q_{i,x}^A}"', from=1-5, to=2-4]
	\arrow["{\cc(A,p_{i,x})}"{description}, from=1-5, to=3-5]
	\arrow[""{name=1, anchor=center, inner sep=0}, "{col^c_W\cc(\alpha,D-)}", from=2-2, to=2-4]
	\arrow["{can_B}"{pos=0.2}, from=2-2, to=3-1]
	\arrow["{can_A}"'{pos=0.2}, from=2-4, to=3-5]
	\arrow[""{name=2, anchor=center, inner sep=0}, "{\cc(\alpha,C)}", from=3-5, to=3-1]
	\arrow[shift left=2, draw=none, from=0, to=1]
	\arrow[shift left, draw=none, from=1, to=2]
\end{tikzcd}\]
where the colimit inclusions in $\Cat$ are denoted by $q^B_{i,x}$ and $q^A_{i,x}$.  Then all regions of the diagram commute.  Arguing as in Part 1, it suffices to show that both quadrilaterals have invertible mates.  The lower one does since $can_A$ and $can_B$ are equivalences; the upper one has identity mate since the colimit inclusions commute strictly with the left adjoints, units and counits of the ralis.  Hence each $p_{i,x}$ is a morphism of Kan injectives.

For the last part of the statement, observe that if $f_i$ are pseudomorphisms of Kan injectives which are jointly liberal in $\cc$ (meaning that the functors $\cc(f_i,X)$ are jointly conservative for each $X$) then each composite $ g  \circ f_i$ is pseudo just when $g$ is pseudo.   Accordingly, it suffices to show that the functors $\cc(p_{i,x},X) \colon \cc(C,X) \to \cc(Di,X)$ are jointly conservative.   But $\cc(p_{i,x},X)$ is equally the composite of the equivalence $\cc(C,X) \to Colax(\ca,\Cat)(W,\cc(D-,X))$ given by composition by $p$ and the evaluation functors $ev_{i,x} \colon Colax(\ca,\Cat)(W,\cc(D-,X)) \to \cc(Di,X)$ which are jointly conservative as before.  As such, the $\cc(p_{i,x},X)$ are jointly conservative, as required.

\item This is the content of Proposition 1.4 of \cite{DLS}.  It is essentially identical in form to Part 1 above, except that one replaces $Colax(\ca,\Cat)$ by $Ps(\ca,\Cat)$ and uses that the right adjoint in $Colax(\ca,\Cat)$ to $\cc(\alpha,D-) \colon \cc(B,D-) \to \cc(A,D-)$ obtained in Part 1 actually belongs to  $Ps(\ca,\Cat)$ --- here, one uses that each $Df \colon Di \to Dj$ is a pseudomorphism of Kan injectives to ensure that the right adjoint components are pseudonatural.
\item As in Part 2 with a similar modification as in Part 3.
\end{enumerate}
\end{proof}

Using this, we now establish the result promised at the beginning of the section.

\begin{theorem}\label{thm:char}
Let $T$ be a pseudomonad on a $2$-category $\cc$  and suppose that $\cc$ admits colax bilimits of arrows.  Then $T$ is lax-idempotent if and only if $U^T$ is locally full and satisfies the colax bilimit property.
\end{theorem}
\begin{proof}
The \emph{if} direction follows from Theorem \ref{thm:laxidempotent} by using the biadjunction $F^T \dashv U^T$.  Conversely, if $T$ is lax-idempotent, we have the $2$-equivalence $ \PTAlg_p \simeq \Kan(\mathcal U)_p$ of Theorem~\ref{thm:semantics} commuting with the forgetful $2$-functors to $\cc$.  As such, it suffices therefore to check that the inclusion $U \colon \Kan(\mathcal U)_p \to \cc$ is locally full and satisfies the colax bilimit property.  It is locally full by definition of $2$-cells in $\Kan(\mathcal U)_p$.  That it satisfies the colax bilimit property is the content of Lemma \ref{prop:kaninjective} Part 1, applied to the special case of the colax bilimit of an arrow.
\end{proof}

\begin{remark} As before, Theorem~\ref{thm:char} admits a pseudo-version, in which lax-idempotent pseudomonads are replaced by pseudo-idempotent ones and colax bilimits are replaced by bilimits.  \end{remark}

\subsection{Monadicity}

In Theorem 4.11 of \cite{CMV}, the authors established a monadicity theorem characterising when a forgetful $2$-functor $U \colon \ca \to \cb$ is monadic for a lax-idempotent pseudomonad.  One feature of this result is that it assumes the induced pseudomonad is lax-idempotent rather than providing a criterion for recognising this in terms of $U$.  It is this point which we will address in Theorem \ref{thm:monadicity2} below, building on our earlier results.

%
%
%

A second feature of Theorem 4.11 of \cite{CMV} is that it assumes the biadjunction $F \dashv U$ involves $2$-functors rather than pseudofunctors, though it is noted in the introduction to \cite{CMV} that their results can be extended to handle pseudofunctors at the cost of the diagrams becoming more complicated.  Since we want a version of Theorem 4.11 allowing for pseudofunctors, we start by making this adaptation in Theorem \ref{thm:monadicity1} below, using a strictification approach.

In order to explain the monadicity theorem of \emph{loc.cit.}, let us firstly explain the terms involved.  Given a biadjunction $F \dashv U$, we have the canonical comparison $E$
\[\begin{tikzcd}
	\ca && {\PTAlg_p} \\
	& \cb
	\arrow["E", from=1-1, to=1-3]
	\arrow["U"', from=1-1, to=2-2]
	\arrow["{U^T}", from=1-3, to=2-2]
\end{tikzcd}\]
to the $2$-category of pseudoalgebras for the induced pseudomonad $T = UF$.  Then $U$ is said to be monadic if $U$ is a biequivalence of $2$-categories.  

In the case that the biadjunction is lax-idempotent, so that $T$ is too, we have the $2$-equivalence $\PTAlg_p \simeq \Kan(\mathcal U)_p$ of Theorem \ref{thm:semantics} commuting with the forgetful $2$-functors to the base.  As such, the canonical comparison $E$ is a biequivalence just when the composite $E^\star$ depicted below is a biequivalence.
\begin{equation}\label{eq:comparison2}
\begin{tikzcd}
	\ca && {\PTAlg_p} && {\Kan(\mathcal U)_p} \\
	&& \cb
	\arrow["E", from=1-1, to=1-3]
	\arrow["{E^\star}", curve={height=-24pt}, from=1-1, to=1-5]
	\arrow["U"', from=1-1, to=2-3]
	\arrow["\simeq", from=1-3, to=1-5]
	\arrow["{U^T}", from=1-3, to=2-3]
	\arrow[hook, from=1-5, to=2-3]
\end{tikzcd}
\end{equation}

%
%

Given a $2$-cell $\alpha$ in $\ca$
\[\begin{tikzcd}
	X && Y && UX && UY & Z & {}
	\arrow[""{name=0, anchor=center, inner sep=0}, "f", curve={height=-18pt}, from=1-1, to=1-3]
	\arrow[""{name=1, anchor=center, inner sep=0}, "g"', curve={height=18pt}, from=1-1, to=1-3]
	\arrow[""{name=2, anchor=center, inner sep=0}, "Uf", curve={height=-18pt}, from=1-5, to=1-7]
	\arrow[""{name=3, anchor=center, inner sep=0}, "Ug"', curve={height=18pt}, from=1-5, to=1-7]
	\arrow["h", from=1-7, to=1-8]
	\arrow["\alpha", between={0.2}{0.8}, Rightarrow, from=0, to=1]
	\arrow["{U\alpha}", between={0.2}{0.8}, Rightarrow, from=2, to=3]
\end{tikzcd}\]

 it is said to have a $U$-absolute bi-coinverter, if its image under $U$ has a bi-coinverter, as depicted above right, which is preserved by every pseudofunctor.\begin{footnote}{Note that \cite{CMV} uses the term pseudo-coinverter for what we here call bi-coinverter.  Since this conflicts with the standard terminology for pseudocolimits, which involve up-to-iso universal properties, we prefer the bicolimit terminology.}\end{footnote} The condition then asks that $\ca$ has a bi-coinverter of those $2$-cells whose image under $U$ has an absolute bi-coinverter, and that $U$ preserves such bi-coinverters.

\begin{theorem}\label{thm:monadicity1}[Proposition 4.7 and Theorem 4.11 of \cite{CMV}]
Let $U \colon \ca \to \cb$ be a pseudofunctor having a left biadjoint $F$.  Then $U$ is monadic for $T=UF$ a lax-idempotent pseudomonad if and only if:
\begin{itemize}
\item The biadjunction is lax-idempotent,
\item $U$ reflects adjoint equivalences and
\item $\ca$ has and $U$ preserves bi-coinverters of $U$-absolute bi-coinverters.
\end{itemize}
\end{theorem}
\begin{proof}
The present theorem combines Proposition 4.7 and Theorem 4.11 of \cite{CMV} with the exception that they assume $U$ and $F$ are $2$-functors rather than pseudofunctors.  We now extend their result to cover pseudofunctors by taking a strictification approach.

Therefore, let us assume that the above three properties hold.  We must prove that $U$ is monadic.   To this end, let $j \colon \twocat_s \to \twocat_p$ be the identity-on-objects inclusion of the category of $2$-categories, $2$-functors and icons into the 2-category of $2$-categories, pseudofunctors and icons \cite{Icons}.
This is the inclusion $j \colon \TAlg_s \to \TAlg_p$ where $T$ is the free enriched category 2-monad on $\Cat$-enriched graphs.   It then follows from \cite{BKP}, as explained in Example \ref{classic}, that the inclusion $j \colon \twocat_s \to \twocat_p$ has a left $2$-adjoint $(-)^\prime$ with unit $p_\ca \colon \ca \to \ca^\prime$ an equivalence in the $2$-category $\twocat_p$.  Now equivalences in $\twocat_p$ are precisely the bijective-on-objects pseudofunctors which are local equivalences and, without loss of generality, we may assume $p_\ca$ is the identity-on-objects.  In particular, it is a biequivalence.  

Applying the left $2$-adjoint to the above biadjunction, we thus get a commuting diagram as below left.   

\[\begin{tikzcd}
	\ca && {\ca^\prime} && {\ca(FX,Y)} && {\cb(X,UY)} \\
	\cb && {\cb^\prime} && {\ca^{\prime}(F^\prime X,Y)} && {\cb^\prime(X,U^{\prime}Y)} \\
	&& {}
	\arrow["{p_\ca}", from=1-1, to=1-3]
	\arrow["U", shift left=3, from=1-1, to=2-1]
	\arrow["{U^\prime}", shift left=3, from=1-3, to=2-3]
	\arrow["{U - \circ \eta_X}", from=1-5, to=1-7]
	\arrow["{p_{\ca}}"', from=1-5, to=2-5]
	\arrow["{p_{\cb}}", from=1-7, to=2-7]
	\arrow["F", shift left=3, from=2-1, to=1-1]
	\arrow["{p_\cb}"', from=2-1, to=2-3]
	\arrow["{F^\prime}", shift left=3, from=2-3, to=1-3]
	\arrow["{U' - \circ \eta'_X}"', from=2-5, to=2-7]
\end{tikzcd}\]

At the unit  $\eta_X:X \to UFX$, let $\eta^{\prime}_X = p_{\cb}(\eta_X) \colon X \to UFX =U^\prime F^\prime X$.  Then the diagram above right commutes strictly.  Since the upper horizontal and two vertical maps are equivalences, the lower horizontal morphism is an equivalence too.  Thus the $\eta^{\prime}_X$ are the unit components of a biadjunction $F^\prime \dashv U^\prime$ whose lax-idempotency follows from that of $F \dashv U$.   

For $T=UF$ and $T^\prime = U^\prime F^\prime$, consider the classes $\mathcal U$ and $\mathcal U^\prime$ of unit maps for the respective biadjunctions, and the induced comparisons 
$E^{\star}_1$ and $E^{\star}_2$ from \eqref{eq:comparison2} as below.

\[\begin{tikzcd}
	& \ca & {\ca^\prime} \\
	& {\Kan(\mathcal U)} & {\Kan(\mathcal U^\prime)} \\
	{} & \cb & {\cb^\prime}
	\arrow["{p_A}", from=1-2, to=1-3]
	\arrow["{E^{\star}_1}"', from=1-2, to=2-2]
	\arrow["U"', curve={height=30pt}, from=1-2, to=3-2]
	\arrow["{E^{\star}_2}", from=1-3, to=2-3]
	\arrow["{U^\prime}", curve={height=-30pt}, from=1-3, to=3-3]
	\arrow["K", from=2-2, to=2-3]
	\arrow[hook, from=2-2, to=3-2]
	\arrow[hook, from=2-3, to=3-3]
	\arrow["{p_B}"', from=3-2, to=3-3]
\end{tikzcd}\]


Since the two remaining properties in the theorem are biequivalence-invariant, $U^\prime$ satisfies them since $U$ does.  Therefore since $U^\prime$ and $F^\prime$ are $2$-functors, we can apply  Theorem 4.11 of \cite{CMV} to conclude that $E^\star_2$ is a biequivalence.

Now observe that since $p_B$ is locally an equivalence, $X \in \cb$ is Kan injective to $f$ just when $X=p_B X$ is Kan injective to $p_B f$ with a similar correspondence holding for morphisms of Kan injectives.  Thus we obtain a factorisation $K \colon \Kan(\mathcal U) \to \Kan(\mathcal U^\prime)$ which is again the identity-on-objects.   It is locally an equivalence since $p_B$ is so and since being a morphism of Kan injectives is invariant under isomorphism.  Thus $K$ is a biequivalence.

Since $p_A$ is a biequivalence and the upper square also commutes, it follows by $2$ from $3$ for biequivalences, that $E^\star_1$ is also a biequivalence, as required.

Conversely, suppose that $U$ is monadic for $T=UF$ lax-idempotent.  Then the biadjunction $F \dashv U$ is lax-idempotent since this is equivalent to $T$'s being lax-idempotent.  Furthermore as $E^\star$ in \eqref{eq:comparison2} is a biequivalence, it suffices to prove that the inclusion $\Kan(\mathcal U)_p \to \cb$ has the remaining two properties.  The inclusion $\Kan(\mathcal U)_p \to \cb$ reflects adjoint equivalences since each equivalence preserve left Kan extensions.  Furthermore, by Proposition \ref{prop:kaninjective} Part 4, $\Kan(\mathcal U)_p$ is closed in $\cb$ under any bicolimits that are absolute in $\cb$, and so $\Kan(\mathcal U)_p \to \cb$ satisfies the condition concerning bi-coinverters.  

\end{proof}

Now here is our refinement of the above result.

\begin{theorem}\label{thm:monadicity2}
Let $U \colon \ca \to \cb$ be a pseudofunctor having a left biadjoint and suppose that $\mathcal B$ admits colax bilimits of arrows.  Then $U$ is bicategorically monadic for $T$ a lax-idempotent pseudomonad if and only if
\begin{itemize}
\item $U$ is locally full,
\item $U$ satisfies the colax bilimit property,
\item $U$ reflects adjoint equivalences and
\item $\ca$ has and $U$ preserves bi-coinverters of $U$-absolute bi-coinverters.
\end{itemize}
\end{theorem}
\begin{proof}
This follows immediately on combining Theorem \ref{thm:char} and Theorem \ref{thm:monadicity1} given that the first two properties are biequivalence-invariant.
%
\end{proof}

\end{document}